\documentclass{article}
\usepackage[utf8]{inputenc}
\pdfoutput=1

\usepackage[margin=1.25in]{geometry}

\usepackage{amsmath,color,graphicx,amssymb}
\usepackage{amsthm}
\usepackage{amsmath}
\usepackage{tikz}
\usetikzlibrary{calc,matrix,arrows,shapes.geometric,positioning,decorations.pathmorphing}
\usepackage{tikz-cd}

\usepackage{float}

\newtheorem{thm}{Theorem}[section] 

\newtheorem{lemma}[thm]{Lemma}     

\newtheorem{cor}[thm]{Corollary}

\newtheorem{prop}[thm]{Proposition}

\theoremstyle{definition}
\newtheorem{defn}[thm]{Definition}

\theoremstyle{definition}
\newtheorem{ex}[thm]{Example}

\theoremstyle{definition}
\newtheorem{remk}[thm]{Remark}

\theoremstyle{definition}
\newtheorem{constr}[thm]{Construction}

\theoremstyle{definition}
\newtheorem{quest}[thm]{Question}

\allowdisplaybreaks

\title{\large{\uppercase{Koszul Homology of Quotients by Edge Ideals}}}
\author{\normalsize{\uppercase{Rachel N. Diethorn}}}
\date{}

\begin{document}

\maketitle

\begin{abstract}
We show that the Koszul homology algebra of a quotient by the edge ideal of a forest is generated by the lowest linear strand.  This provides a large class of Koszul algebras whose Koszul homology algebras satisfy this property.  We obtain this result by constructing the minimal graded free resolution of a quotient by such an edge ideal via the so called iterated mapping cone construction and using the explicit bases of Koszul homology given by Herzog and Maleki.  Using these methods we also recover a result of Roth and Van Tuyl on the graded Betti numbers of quotients of edge ideals of trees.     
\end{abstract}

\section{Introduction}
\paragraph
\indent  Let $k$ be a field and let $R=\bigoplus_{i\geq 0}R_i$ be a standard graded $k$-algebra.  Let $K(R)$ be the Koszul complex on a minimal set of generators of $R_1$.  It is well-known that the differential graded algebra structure on $K(R)$ induces a $k$-algebra structure on its homology, $H(R)$, see for example \cite[1.3]{MR2641236}.  This algebra structure on Koszul homology holds important information about the ring $R$.  For example, $R$ is a complete intersection if and only if $H(R)$ is generated by $H_1(R)$ as a $k$-algebra \cite[Thm 2.7]{assmus}, $R$ is Gorenstein if and only if $H(R)$ satisfies Poincare duality \cite{avramovgolod}, and $R$ is Golod if and only if $K(R)$ admits a trivial Massey operation \cite[Thm 5.2.2]{MR2641236}.  
\newline
\indent
Another property of $R$ that has strong connections to the structure of $H(R)$ is the Koszul property.  $R$ is said to be \textit{Koszul} if $k$ has a linear resolution over $R$.  To discuss the connections between $R$ and $H(R)$ when $R$ is Koszul, one views the Koszul homology algebra $H(R)=\bigoplus_{i,j}H_i(R)_j$ as a bigraded algebra, where $i$ is the homological degree and $j$ is the internal degree given by the grading on $R$.  If $R$ is Koszul, then it is known that $H_i(R)_j=0$ for all $j>2i$ \cite[Thm 3.1]{MR2644369}, that $H_i(R)_{2i}=(H_1(R)_2)^i$ \cite[Thm 5.1]{MR3302628}, and that $H_i(R)_{2i-1}=(H_1(R)_2)^{i-2}H_2(R)_3$ \cite[Thm 3.1]{boocheretal}.  These results show that certain parts of $H(R)$ are generated by the lowest linear strand when $R$ is Koszul.  Avramov asked the following question regarding this behavior.

\begin{quest}\label{quest}
If $R$ is Koszul, is the Koszul homology algebra of $R$ generated as a $k$-algebra by the lowest linear strand?  That is, is $H(R)$ generated by $\bigoplus_i H_i(R)_{i+1}$?
\end{quest}

The answer to this question is negative in general.  The authors of \cite{boocheretal} show that the Koszul homology algebra of the quotient by the edge ideal of an $n$-cycle where $n\equiv 1(\text{mod}\,3)$ is not generated by the lowest linear strand.  However, interest lies in determining for which Koszul algebras, this question has a positive answer.  The answer is positive for the Koszul homology algebra of the quotient by the edge ideal of an $n$-path \cite[Thm 3.15]{boocheretal}   and for the Koszul homology algebra of the second Veronese algebra \cite[Cor 2.4]{2017arXiv171004293C}.  Still the question remains open for many classes of algebras known to be Koszul.

In this paper, we give a positive answer to this question for a large class of edge ideals.  Let $Q=k[x_1,...,x_n]$ be a standard graded polynomial ring over any field $k$ and let $I$ be an edge ideal associated to a tree.  We show that the Koszul homology algebra of the quotient $R=Q/I$ is generated by the lowest linear strand.  This result extends easily to edge ideals of forests and our result recovers \cite[Thm 3.15]{boocheretal}.  To obtain this result, we utilize the so called iterated mapping cone construction and the explicit $k$-bases of each $H_i(R)$ given by Herzog and Maleki in \cite[Thm 1.3]{Herzog2018}.

We now outline the contents of this paper.  In Section 2, we recall some important terminology which we use throughout the paper and we discuss the main tools we use in our results, including the iterated mapping cone construction, multiplicative structures on resolutions, and explicit bases for the Koszul homology modules from \cite{Herzog2018}.  In Section 3, we construct the minimal graded free resolution of $Q/I$ over $Q$ which we use in the proof of our main result.  We also recover a result of Roth and Van Tuyl in \cite{rothvantuyl} on the Betti numbers of such quotients $Q/I$.  In Section 4, we state and prove the main result.

\section{Preliminaries}
\paragraph
\indent  In this section, we set up the basic terminology which we use throughout the paper and discuss the main tools we use to obtain our results.  Let $Q=k[x_1,...,x_n]$ be a standard graded polynomial ring over a field $k$.  
\newline
\indent  We begin by recalling the notion of graded Betti numbers.  We consider the minimal graded free resolution $F$
\begin{align*}
...\longrightarrow\underset{j}{\bigoplus}Q(-j)^{\beta_{i,j}}\longrightarrow\underset{j}{\bigoplus}Q(-j)^{\beta_{i-1,j}}\longrightarrow...\longrightarrow\underset{j}{\bigoplus}Q(-j)^{\beta_{0,j}}
\end{align*}
of a $Q$-module $M$.  The \textit{$i$-th graded Betti number of internal degree $j$} is $\beta_{i,j}$.  The \textit{Betti table} of $M$ is given by
\begin{center}
\begin{tabular}{ |c|c c c c c  } 
 \hline
 \,\, & 0 & 1 & 2 & 3 & ...\\ 
 \hline
 0 & $\beta_{0,0}$ & $\beta_{1,1}$ & $\beta_{2,2}$ & $\beta_{3,3}$ & ... \\
 1 & $\beta_{0,1}$ & $\beta_{1,2}$ & $\beta_{2,3}$ & $\beta_{3,4}$ & ... \\ 
 2 & $\beta_{0,2}$ & $\beta_{1,3}$ & $\beta_{2,4}$ & $\beta_{3,5}$ & ... \\ 
 3 & $\beta_{0,3}$ & $\beta_{1,4}$ & $\beta_{2,5}$ & $\beta_{3,6}$ & ... \\ 
 \vdots & \vdots & \vdots & \vdots & \vdots & 
\end{tabular}
\end{center}
\indent Now we recall the following basic isomorphism, which we use throughout this paper.  Let $I$ be a homogeneous ideal in $Q$ and let $R=Q/I$.  Throughout this paper, we denote the homology of the Koszul complex $K(x_1,...,x_n;R)$ by $H(R)$.  If $F$ is the minimal graded free resolution of $R$ over $Q$, then there is an isomorphism of $k$-algebras
\begin{align}\label{eq:6}
\Phi:F\otimes k\rightarrow H(R).
\end{align}
Thus, given a basis $e_1^i,...,e_{b_i}^i$ of $F_i$, we have that the elements $\Phi(e_j^i\otimes\bar{1})$ for $j=1,...,b_i$,  form a basis for $H_i(R)$.  Furthermore, if deg\,$e_j^i=k$, then $\Phi(e_j^i\otimes\bar{1})\in H_i(R)_{k}$.  Given this isomorphism we can represent the Koszul homology algebra of $R$ as a table
\begin{center}
\begin{tabular}{ |c|c c c c c  } 
 \hline
 \,\, & 0 & 1 & 2 & 3 & ...\\ 
 \hline
 0 & $H_{0,0}$ & $H_{1,1}$ & $H_{2,2}$ & $H_{3,3}$ & ... \\
 1 & $H_{0,1}$ & $H_{1,2}$ & $H_{2,3}$ & $H_{3,4}$ & ... \\ 
 2 & $H_{0,2}$ & $H_{1,3}$ & $H_{2,4}$ & $H_{3,5}$ & ... \\ 
 3 & $H_{0,3}$ & $H_{1,4}$ & $H_{2,5}$ & $H_{3,6}$ & ... \\ 
 \vdots & \vdots & \vdots & \vdots & \vdots & 
\end{tabular}
\end{center}
where $H_{i,j}=H_i(R)_j$.  In this paper we often discuss the \textit{lowest linear strand} of $H(R)$, which is the second row (i.e. row 1) in the table above.

\subsection{Edge Ideals}
\paragraph
\indent Let $Q=k[x_1,...,x_n]$ be a standard graded polynomial ring over a field $k$.  We begin this subsection by recalling the notion of an edge ideal. 

\begin{defn}
Let $G$ be a simple graph (that is, with no loops nor multiple edges) on vertices $x_1,...,x_n$.  The \textit{edge ideal} associated to $G$ is the ideal
\begin{align*}
I_G=(x_ix_j\,|\,x_ix_j \,\text{is an edge in}\,\, G). 
\end{align*}
\end{defn} 

If $G$ is a graph on the variables of $Q$ and $G'$ is a subgraph of $G$, we write $I_{G'}$ for the edge ideal associated to $G'$ in $Q$.  The class of edge ideals we focus on in this paper is that of trees.  

\begin{defn}
Let $G$ be a simple graph.  $G$ is a \textit{tree} if $G$ is connected and contains no cycle.  Equivalently, $G$ is a \textit{tree} if every vertex in $G$ is connected by exactly one path.  A \textit{leaf} is a vertex in $G$ of degree $1$.  A \textit{forest} is a disjoint union of trees and a \textit{subforest} of a forest $G$ is a subgraph of $G$ which is a forest.  
\end{defn}

We illustrate the above definitions with the following example.

\begin{ex}
Let $Q=k[x_1,x_2,x_3,x_4,x_5,x_6,x_7]$ be a polynomial ring.  The edge ideal associated to the tree $G$ shown in Figure \ref{fig1} is $I_G=(x_1x_2,\,x_2x_3,\,x_2x_4,\,x_2x_5,\,x_3x_6,\,x_4x_7)$.
\end{ex}

\begin{figure}[H]
\vspace{-.5cm}
\centering
\includegraphics[width=6cm]{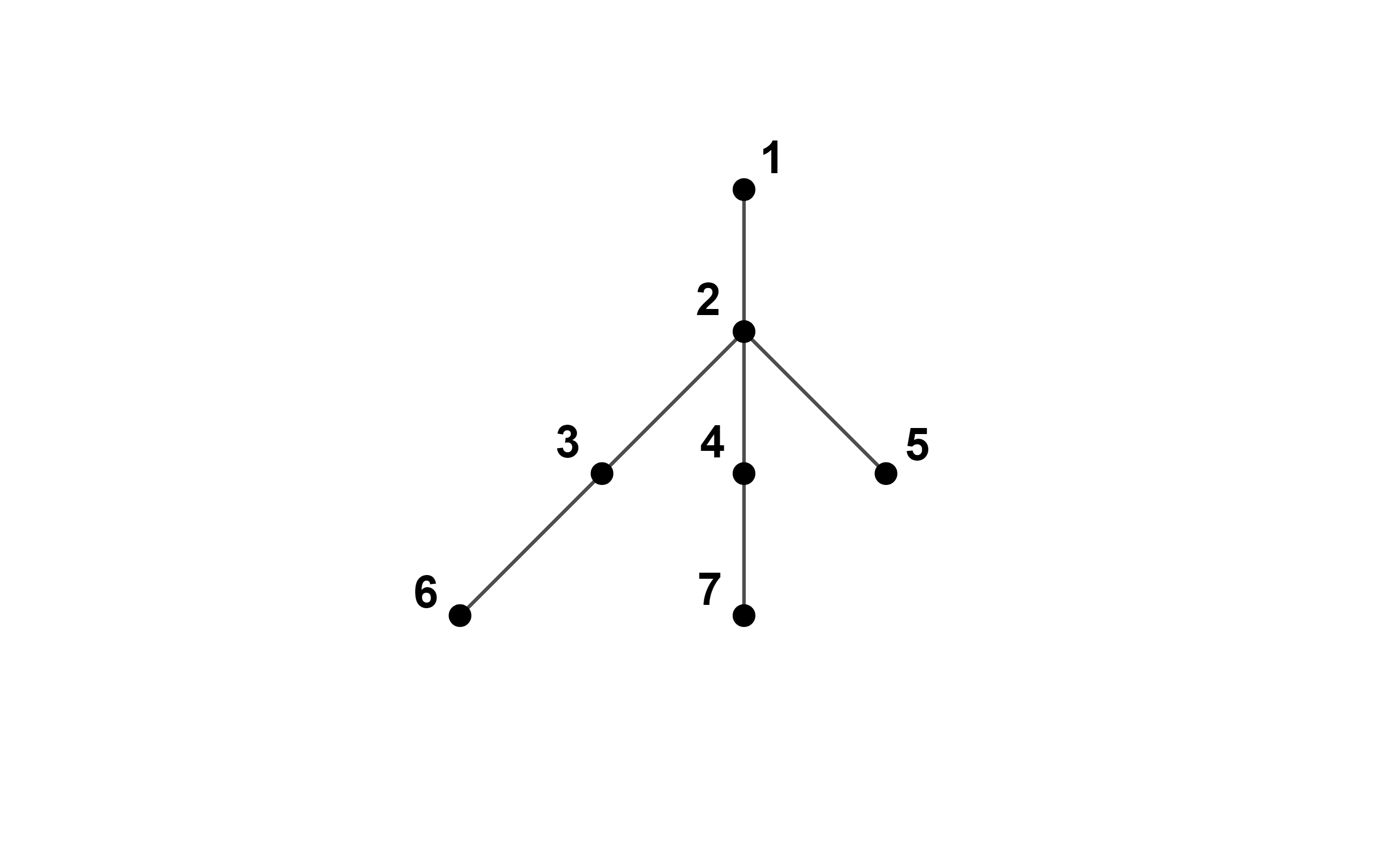}
\vspace{-.75cm}
\caption{The graph $G$ is a tree.}
\label{fig1}
\end{figure}

We make the following easy remarks about trees that will be useful throughout this paper.

\begin{remk}\label{treeremk}\hfill
\begin{itemize}
\item[(i)]  By definition, a tree $G$ must have a leaf, otherwise $G$ would contain a cycle.
\item[(ii)]  It is easy to see that any subgraph of a tree is a subforest.  
\end{itemize}
\end{remk}

In the following subsection, we discuss a way to obtain the minimal graded free resolution of a quotient by the edge ideal of a tree.

\subsection{Iterated Mapping Cones}

In this subsection, we discuss one of the main tools we use to obtain the results in this paper, namely the iterated mapping cone construction.  We begin by recalling the notion of a mapping cone.

\begin{defn}
Let $(F,\partial^F)$ and $(G,\partial^G)$ be two complexes of finitely generated $Q$-modules and let $\phi:F\rightarrow G$ be a map of complexes.  The \textit{mapping cone} of $\phi$, denoted $\text{cone}\,(\phi)$, is the complex $(\text{cone}\,(\phi),\partial)$ with
\begin{align*}
(\text{cone}\,(\phi))_i=G_i\oplus F_{i-1} \\
\partial_i=
\left[
\begin{array}{c c}
\partial_i^G & \phi_{i-1} \vspace{.12cm}\\
0 & -\partial_{i-1}^F \\
\end{array}
\right].
\end{align*}
\end{defn}

It is easy to see the following fact.

\begin{remk}\label{remksubcx}
If $\phi:F\rightarrow G$, then there is a short exact sequence
\begin{align*}
0\longrightarrow G\longrightarrow\mathrm{cone}(\phi)\longrightarrow F[-1]\longrightarrow 0.
\end{align*}
Thus, $G$ is a subcomplex of $\text{cone}\,(\phi)$. 
\end{remk}

Mapping cones can be used to build free resolutions of quotients by monomial ideals in the following way.  See, for example, \cite[Constr 27.3]{peeva} for more details.

\begin{constr}\label{coneconstr}
Let $Q$ be a graded polynomial ring, and let $I$ be the ideal minimally generated by monomials $m_1,...,m_r$.  Denote by $d_i$ the degree of the monomial $m_i$ and by $I_i$ the ideal generated by $m_1,...,m_i$.  For each $i\geq 1$, we have the following graded short exact sequence
\begin{align*}
0\longrightarrow Q/(I_i:m_{i+1})(-d_{i+1})\overset{m_{i+1}}{\longrightarrow}Q/I_i\longrightarrow Q/I_{i+1}\longrightarrow 0.
\end{align*}
Note that we have shifted the first term by the degree of the monomial $m_{i+1}$ to make the multiplication by $m_{i+1}$ a degree zero map.  Thus, given graded $Q$-free resolutions $G^i$ of $Q/I_i$ and $F^{i}$ of $Q/(I_i:m_{i+1})$, there is a map of complexes $\phi_i:F^i\rightarrow G^i$ induced by multiplication by $m_{i+1}$, which we will call the \textit{comparison map}.  The mapping cone of the comparison map is a graded free resolution $F^{i+1}=\text{cone}\,(\phi_i)$ of $Q/I_{i+1}$.  Applying this construction for each $i=1,...,r-1$ to obtain a graded free resolution of $Q/I=Q/I_r$ is called the \textit{iterated mapping cone construction}. 
\end{constr}

We make the following important remarks about the iterated mapping cone construction.

\begin{remk} \label{remkconeconstr}Using the notation from Construction \ref{coneconstr}, we note the following. \hfill
\begin{itemize}
\item[(i)]  The resolution of $I=(m_1,...m_r)$ produced by the mapping cone construction depends on the given order of the monomials.  We illustrate this remark in Example \ref{exorder} below.
\item[(ii)] For any $i\geq 1$, $\text{cone}\,(\phi_i)$ need not be minimal, even if the given free resolutions $F^i$ and $G^i$ are minimal.  Thus, the resolution of $I$ produced by the iterated mapping cone construction need not be minimal.  We illustrate this remark in Example \ref{cycleex} below.    
\end{itemize}
\end{remk}

We now recall the following theorem that follows from results of H{\`a} and Van Tuyl in \cite{havan} and was proved independently by Bouchat in \cite[Thm 3.0.16]{MR2992613}. It will be useful in the proofs of our results.  

\begin{thm}\label{bouchat}
Let $Q=k[x_1,...,x_n]$ and let $G$ be a simple graph on vertices $x_1,...,x_n$ such that $x_n$ is a vertex of degree $1$ and is connected by an edge to the vertex $x_{n-1}$.  Then the mapping cone construction applied to the map 
\begin{align*}
Q/(I_{G\setminus{x_n}}:x_{n-1}x_n)(-2)\overset{x_{n-1}x_n}{\longrightarrow}Q/I_{G\setminus x_n}
\end{align*}
is a minimal graded free resolution of $Q/I_G$.   
\end{thm}

The following example shows that the conclusion of Theorem \ref{bouchat} need not hold if the graph $G$ has no vertex of degree 1.

\begin{ex}\label{cycleex}
Let $G$ be the 5-cycle shown in Figure \ref{fig2} and consider its associated edge ideal $I_G=(x_1x_2,\,x_2x_3,\,x_3x_4,\,x_4x_5,\,x_1x_5)$.  Applying the iterated mapping cone construction to resolve $Q/I_G$, we get the following comparison map in the last iteration. \\
\begin{tikzcd}[ampersand replacement=\&]
0 \arrow{r} \& 0 \arrow{r} \& Q(-4) \arrow{ddd}{\begin{bmatrix}
  0\\
  0\\
  0\\
  -1
\end{bmatrix}} \arrow{rrrr}{
\begin{bmatrix}
  -x_4\\
  x_2
\end{bmatrix}
} \&\&\&\& Q(-3)^2 \arrow{rrrr}{
\begin{bmatrix}
  x_2 & x_4
\end{bmatrix}} \arrow{ddd}{\begin{bmatrix}
  x_5 & 0\\
  0& 0\\
  0& 0\\
  0 & x_1
\end{bmatrix}} \&\&\&\& Q(-2) \arrow{ddd}{x_1x_5} \\\\\\
0 \arrow{r} \& Q(-5) \arrow{r}{}[swap]{\begin{bmatrix}
  x_4x_5\\
  x_1x_5\\
  x_1x_2\\
  -x_3
\end{bmatrix}} \& Q(-3)^3\oplus Q(-4) \arrow{rrrr}{}[swap]{\begin{bmatrix}
  x_3 & 0 & 0 & x_4x_5\\
  -x_1 & x_4 & 0 & 0\\
  0 & -x_2 & x_5 & 0\\
  0 & 0 & -x_3 & -x_1x_2
\end{bmatrix}} \&\&\&\& Q(-2)^4 \arrow{rrrr}{}[swap]{\begin{bmatrix}
  x_1x_2 & x_2x_3 & x_3x_4 & x_4x_5
\end{bmatrix}} \&\&\&\& Q \\
\end{tikzcd} 
\newline \hfill
\newline \hfill
\newline \hfill
We see that the cone of this comparison map will produce a non-minimal resolution.
\end{ex}

\begin{figure}[H]
\vspace{-.5cm}
\centering
\includegraphics[width=6cm]{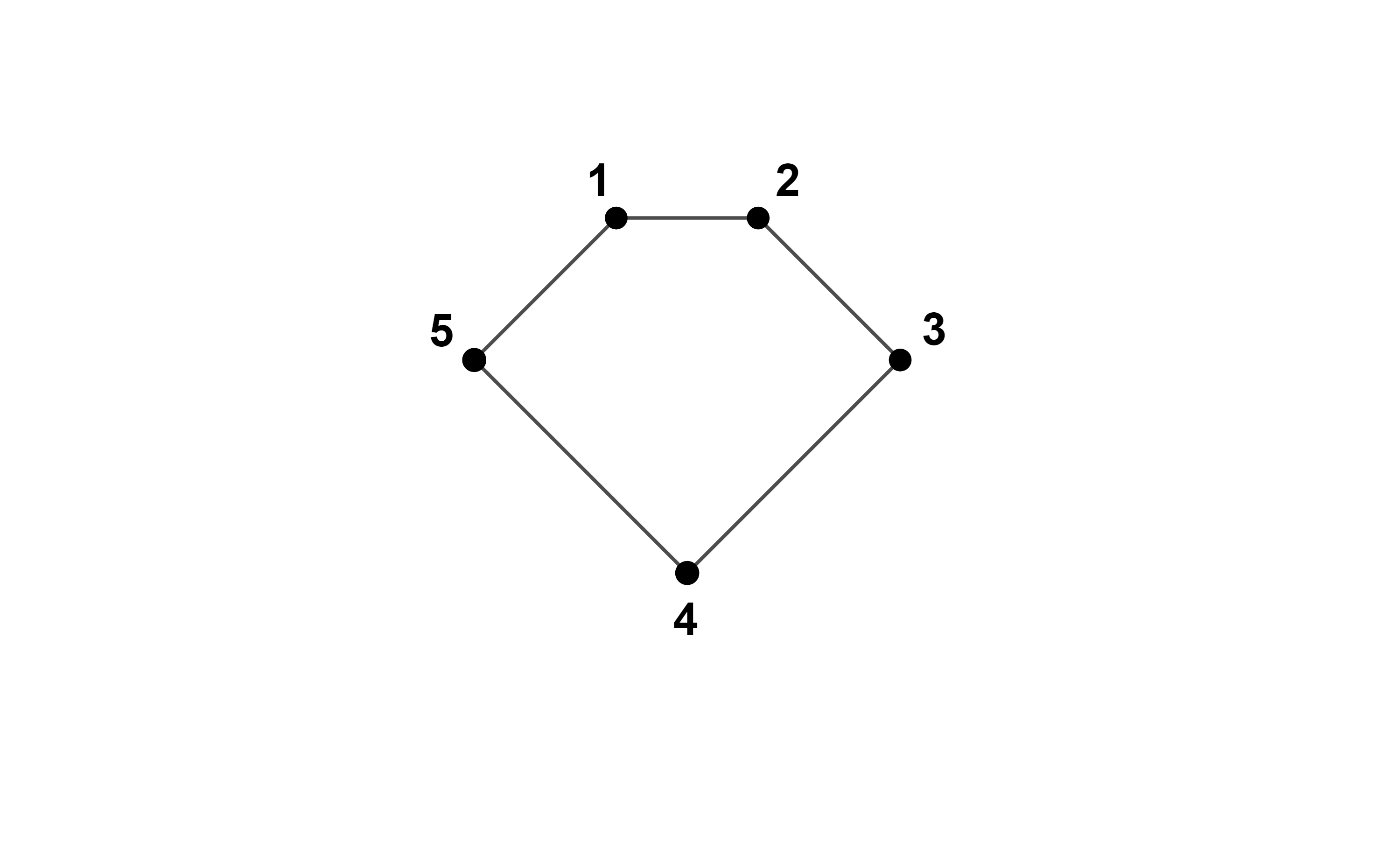}
\vspace{-.75cm}
\caption{The graph $G$ is a 5-cycle.}
\label{fig2}
\end{figure}

By Remark \ref{treeremk}, Theorem \ref{bouchat} provides an inductive method for finding the minimal graded free resolution of $Q/I_G$, where $G$ is any tree.  We state this as a corollary.

\begin{cor}\label{corbouchat}
If $G$ is a tree, then, in some order, the iterated mapping cone construction gives the minimal graded free resolution of $Q/I_G$ over $Q$.  
\end{cor}

The following example illustrates the importance of the order in which the iterated mapping cone construction is applied.

\begin{ex}\label{exorder}
Let $G$ be the tree shown in Figure \ref{fig3} and consider its associated edge ideal $I_G=(x_1x_3,\,x_2x_3,\,x_3x_4,\,x_4x_5)$.  Applying the iterated mapping cone construction to resolve $Q/I_G$, we get the following comparison map in the last iteration. 
\begin{center}
\begin{tikzcd}[ampersand replacement=\&]
0 \arrow{r} \& 0 \arrow{r} \& 0 \arrow{rrr} \&\&\& Q(-3) \arrow{rrr}{
x_3} \arrow{dd}{\begin{bmatrix}
  0 \\
  0 \\
  x_5  
\end{bmatrix}} \&\&\& Q(-2) \arrow{dd}{x_4x_5} \\\\
0 \arrow{r} \& Q(-4) \arrow{r}{}[swap]{\begin{bmatrix}
 -x_4\\
 x_2\\
-x_1
\end{bmatrix}} \& Q(-3)^3 \arrow{rrr}{}[swap]{\begin{bmatrix}
  x_2 & x_4 & 0 \\
  -x_1 & 0 & x_4 \\
  0 & -x_1 & -x_2
\end{bmatrix}} \&\&\& Q(-2)^3 \arrow{rrr}{}[swap]{\begin{bmatrix}
  x_1x_3 & x_2x_3 & x_3x_4
\end{bmatrix}} \&\&\& Q \\
\end{tikzcd}
\end{center}
However, if we instead order the minimal generators of the edge ideal as $I_G=(x_1x_3,\,x_2x_3,\,x_4x_5,\,x_3x_4)$ and apply the iterated mapping cone construction, we get the following comparison map in the last iteration. 
\begin{center}
\begin{tikzcd}[ampersand replacement=\&]
0 \arrow{r} \& Q(-5) \arrow{r}{\begin{bmatrix}
 x_5\\
 -x_2\\
x_1
\end{bmatrix}}\arrow{dd}{-1} \& Q(-4)^3 \arrow{rrrr}{\begin{bmatrix}
  -x_2 & -x_5 & 0 \\
  x_1 & 0 & -x_5 \\
  0 & x_1 & x_2
\end{bmatrix}}\arrow{dd}{\begin{bmatrix}
  -x_4 & 0 & 0 \\
  0 & -1 & 0 \\
  0 & 0 & -1
\end{bmatrix}} \&\&\&\& Q(-3)^3 \arrow{rrr}{\begin{bmatrix}
  x_1 & x_2 & x_5
\end{bmatrix}} \arrow{dd}{\begin{bmatrix}
  x_4 & 0 & 0 \\
  0 & x_4 & 0 \\
  0 & 0 & x_3
\end{bmatrix}} \&\&\& Q(-2) \arrow{dd}{x_3x_4} \\\\
0 \arrow{r} \& Q(-5) \arrow{r}{}[swap]{\begin{bmatrix}
 x_4x_5\\
 -x_2\\
x_1
\end{bmatrix}} \& Q(-3)\oplus Q(-4)^2 \arrow{rrrr}{}[swap]{\begin{bmatrix}
  x_2 & x_4x_5 & 0 \\
  -x_1 & 0 & x_4x_5 \\
  0 & -x_1x_3 & -x_2x_3
\end{bmatrix}} \&\&\&\& Q(-2)^3 \arrow{rrr}{}[swap]{\begin{bmatrix}
  x_1x_3 & x_2x_3 & x_4x_5
\end{bmatrix}} \&\&\& Q \\
\end{tikzcd}
\end{center}
It is clear that applying the mapping cone construction in these two orders produce different resolutions.  We note that the second resolution is not minimal.  If it was, it would have to be isomorphic to the first one.
\end{ex}

\begin{figure}[H]
\vspace{-.5cm}
\centering
\includegraphics[width=6cm]{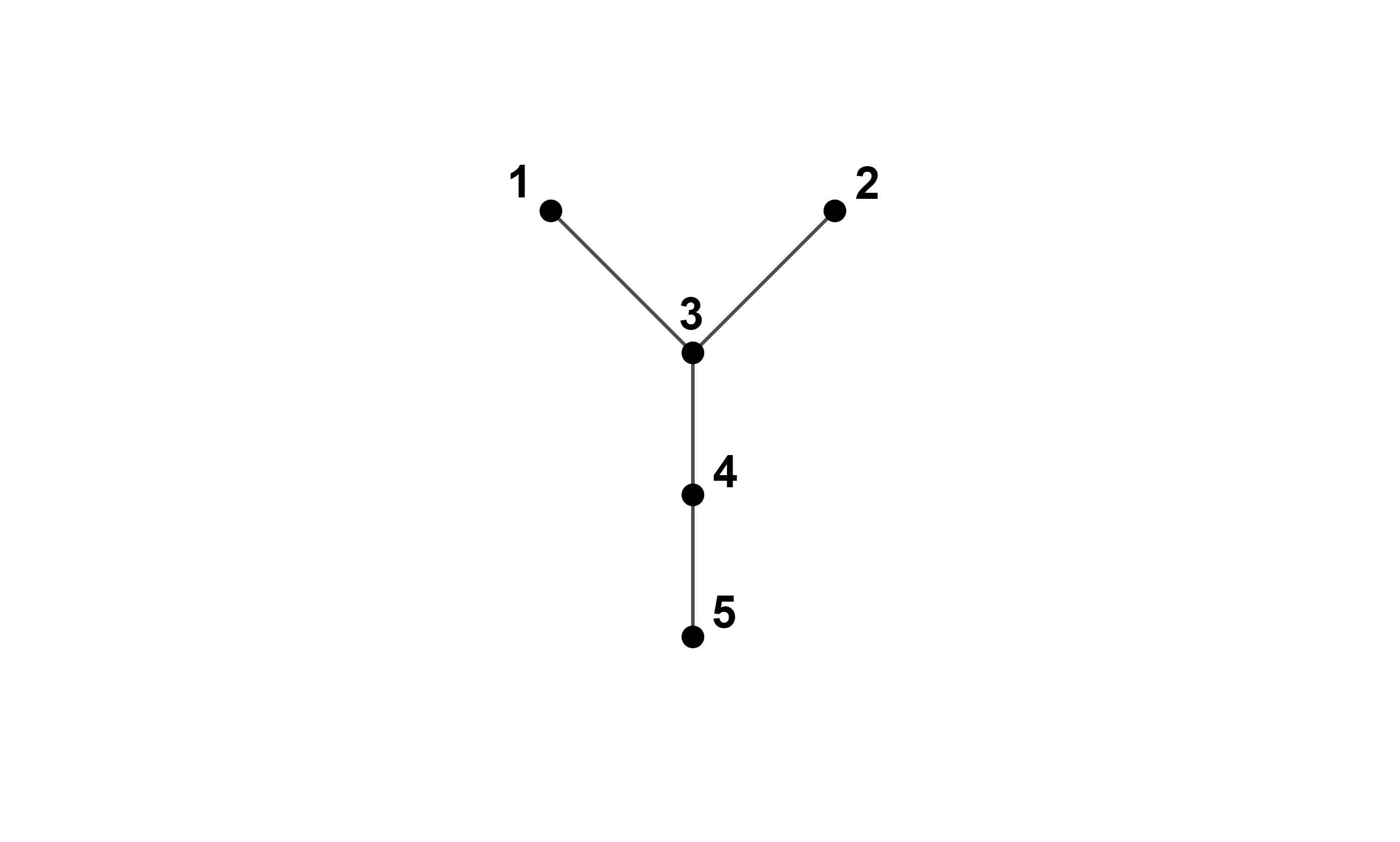}
\vspace{-.75cm}
\caption{The graph $G$ is a tree}
\label{fig3}
\end{figure}

We use the iterated mapping cone construction in Section 3 to explicitly build the minimal graded free resolution of $Q/I_G$, where $G$ is a tree.  This resolution is an important ingredient in our proof of the main result in Section 4.

\subsection{Multiplicative Structures on Resolutions} 
\paragraph
\indent
Let $Q=k[x_1,...x_n]$ be a standard graded polynomial ring over any field $k$ and let $I$ be a monomial ideal of $Q$.  Let $F$ be the minimal graded free resolution of $Q/I$ over $Q$.  In this section we recall the notion of a multiplicative structure on $F$; see for example \cite{katthan}.

\begin{defn}
A $Q$-linear map $F\otimes_Q F\rightarrow F$ sending $a\otimes b$ to $a\cdot b$ is a \textit{multiplication} if it satisfies the following conditions for $a,b\in F$:
\begin{itemize}
\item[(i)] it extends the usual multiplication on $F_0=Q$
\item[(ii)]it satisfies the Leibniz rule: $\partial(ab)=\partial(a)b+(-1)^{|a|}a\partial(b)$
\item[(iii)]  it is homogeneous with respect to the homological grading: $|a\cdot b|=|a|+|b|$
\item[(iv)] it is graded commutative: $a\cdot b=(-1)^{|a||b|}b\cdot a$
\end{itemize}  
\end{defn}

Notice we do not require that a multiplication is associative. The following fact is due to Buchsbaum and Eisenbud.

\begin{prop}\cite[Prop 1.1]{BE77}\label{propbe}
The resolution $F$ admits a multiplication.
\end{prop}

This fact will be useful in the proofs of our results.

\subsection{Explicit Bases for Koszul Homology} 
\paragraph
\indent Let $Q=k[x_1,...,x_n]$ be a standard graded polynomial ring over any field $k$, let $I$ be a homogeneous ideal of $Q$, and let $R=Q/I$.  In this section, we discuss explicit bases of the Koszul homology modules $H_i(R)$ given by Herzog and Maleki in \cite{Herzog2018}.  In order to describe these bases explicitly, we first set up some notation.
\newline
\indent Herzog and Maleki define operators on $Q$ as follows.  For $f\in(x_1,...,x_n)$ and for $r=1,...,n$, let
\begin{align*}
d^r(f)=\frac{f(0,...,0,x_r,...,x_n)-f(0,...0,x_{r+1},...x_n)}{x_r}.
\end{align*}
It is clear that the operators $d^r:Q\rightarrow Q$ are $k$-linear maps and that they depend on the order of the variables.  In this paper, we apply these operators to monomials.  The following basic lemma describes how $d^r$ behaves in this context.  

\begin{lemma}\label{lemmaoperators}
Let $f$ be the monomial $x_{k_1}...x_{k_i}$ with $k_1\leq...\leq k_i$.  Then
\begin{align*}
d^r(f)=\begin{cases}
x_{k_2}...x_{k_i} & r=k_1 \\
0 & \text{otherwise}
\end{cases}.
\end{align*}
\end{lemma}

\begin{proof}
If $r=k_1$, then by definition $d^r(f)=\frac{x_{k_1}...x_{k_i}-0}{x_{k_1}}=x_{k_2}...x_{k_i}$.  If $r< k_1$, then $d^r(f)=\frac{f-f}{x_r}=0$.  If $r> k_1$, then $d^r(f)=\frac{0-0}{x_r}=0$
\end{proof}

This simple fact will be useful in the proof of our main result.  The following theorem due to Herzog and Maleki describes explicit bases for the Koszul homology modules.  To set notation for the theorem, let $F$ be a minimal graded free resolution of $Q/I$ over $Q$ and let $b_i$ be the rank of $F_i$ for each $i$.  For each $i$, fix a homogeneous basis $e_1^i,...,e_{b_i}^i$ of $F_i$ and let $\partial(e_j^i)=\sum_{k=1}^{b_{i-1}} f_{k,j}^i e_k^{i-1}$.  Let $dx_1,...,dx_n$ be the standard generators of $K(x_1,...,x_n;Q)$.  

\begin{thm}\cite[Thm 1.3]{Herzog2018}\label{thmhm}
For each $i=1,...,n$, a $k$-basis of $H_i(R)$ is given by $[\bar{z}_{i,j}]$ for $j=1,...,b_i$, where 
\begin{align*}
z_{i,j}=\sum_{1\leq k_1<...<k_i\leq n}\sum_{j_1=1}^{b_1}...\sum_{j_{i-1}=1}^{b_{i-1}}d^{k_i}(f_{j_{i-1},j}^i)...d^{k_2}(f_{j_1,j_2}^2)d^{k_1}(f_{1,j_1}^1)dx_{k_1}...dx_{k_i}.
\end{align*}
\end{thm} 

In the proof of Theorem \ref{thmhm}, Herzog and Maleki show that the isomorphism (\ref{eq:6}) is given explicitly by
\begin{align*}
&\Phi:F\underset{Q}{\otimes}k\rightarrow H(R) \\
&\Phi(e_{j}^i\otimes \bar{1})=[\bar{z}_{i,j}].
\end{align*}

We conclude this section with the following remark.

\begin{remk}
We note that in \cite{MR1310371}, Herzog gives a different description of bases for $H_i(R)$ under the assumption that $k$ is a field of characteristic zero.  In \cite{2019arXiv190408350D}, the author gives explicit bases in a more general case, namely the case of the Koszul complex on any full regular sequence, which recover Herzog's bases in characteristic zero.  In this paper, we use the description of the bases given by Herzog and Maleki as they hold in any characteristic and also many terms vanish in the case of a monomial ideal.  This vanishing plays an important role in the proof of our main result.     
\end{remk}

\section{Construction of the Resolution}
\paragraph
\indent  Let $G$ be a tree and let $Q$ be a standard graded polynomial ring over any field $k$ with variables given by the vertices in $G$.  In this section we construct the minimal graded free resolution of $Q/I_G$ over $Q$.  We begin by setting the notation to be used throughout the section.  
\newline
\indent We name the vertices in $G$ as follows.  Since $G$ is a tree, it has at least one vertex of degree 1; call it $x_1$ and call the vertex it is connected to by an edge $x_2$.  Call the other vertices to which $x_2$ is connected to by an edge, $x_{2,1},...,x_{2,r}$.  For each $\ell=1,...,r$, call the other vertices to which $x_{2,\ell}$ is connected to by an edge, $x_{2,\ell,1},...,x_{2,\ell,m_{\ell}}$.  Note that $G\smallsetminus\{x_1,x_2,x_{2,1},...,x_{2,r}\}$ is a subforest.  In particular, it is the disjoint union of $M:=\sum_{\ell=1}^rm_{\ell}$ trees, call them $T_1,...,T_{M}$.  With this notation in mind, we view $G$ as the following diagram. \\      

\begin{figure}[H]
\vspace{-.5cm}
\centering
\includegraphics[width=8.5cm]{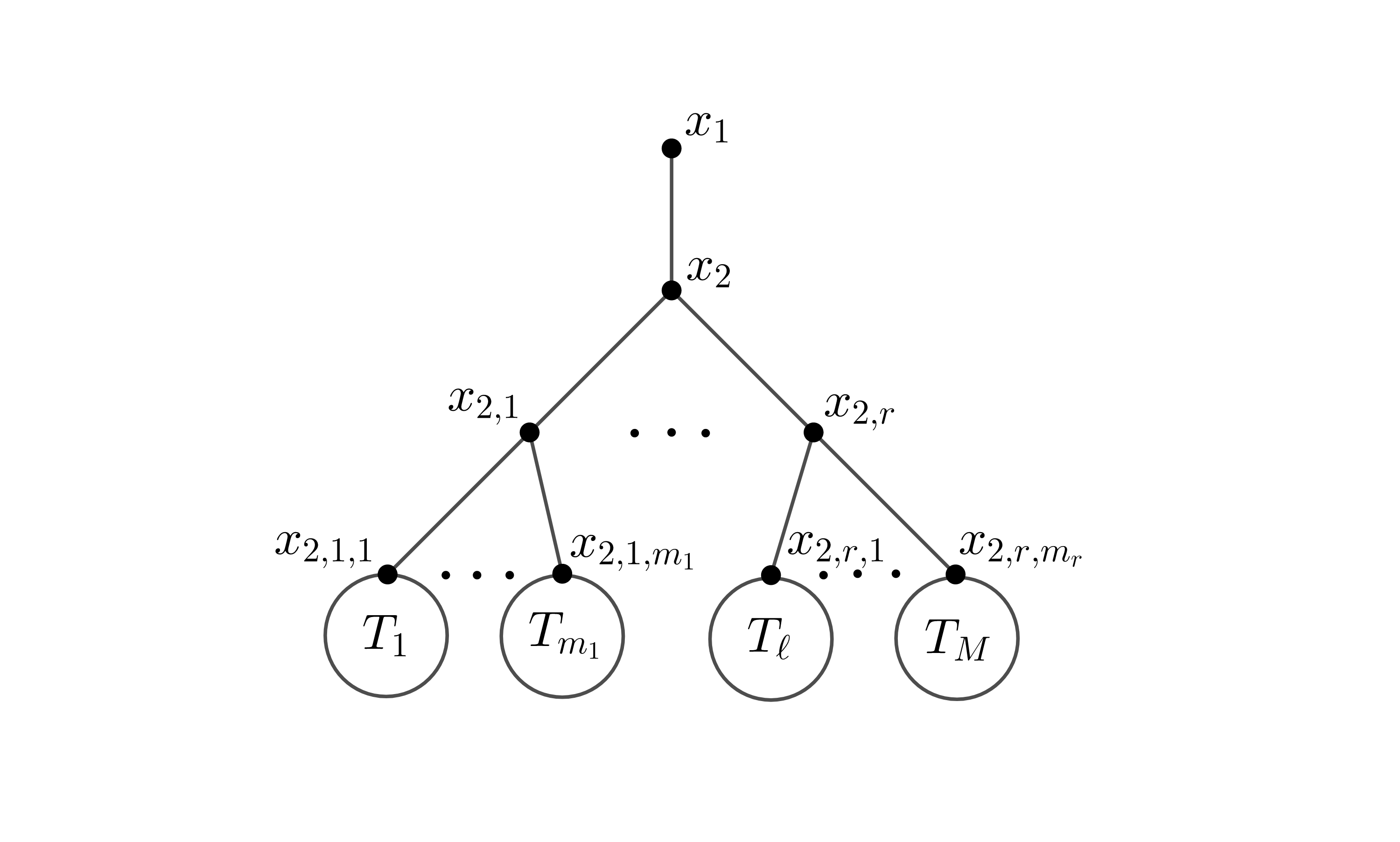}
\vspace{-.75cm}
\caption{The graph $G$ with $\ell=M-r_m-1$}
\label{fig6}
\end{figure}

\indent  We aim to resolve $Q/I_G$ minimally.  By Theorem \ref{bouchat}, this can be done by applying the iterated mapping cone construction as long as at each iteration we add a vertex of degree one.  We choose the following order to apply the iterated mapping cone construction:
\begin{align}\label{eq:4}
I_G=(x_2x_{2,1},...,x_2x_{2,r},\{x_{2,\ell}x_{2,\ell,p}{\}}_{\substack{\ell=1,...r \\ p=1,...,m_{\ell}}}, e(T_1),...e(T_{M}),x_1x_2)
\end{align}
where, abusing notation, we write $e(T_i)$ to mean the set of relations coming from the edges of the tree $T_i$, taking the edge connecting $T_i$ to the corresponding $x_{2,\ell,p}$ to be the first one.  By Corollary \ref{corbouchat} there is an ordering for each $T_i$ which will preserve minimality in the iterated mapping cone construction; we choose such an ordering for each one.  In this way, we obtain the minimal graded free resolution of $Q/I_G$.  Throughout the remainder of this section, we will write this resolution more explicitly in order to obtain our main result in the next section.
\newline
\indent  Denote by $G_1$ the graph $G\smallsetminus x_1$ and by $C$ the colon ideal $(I_{G_1}:x_1x_2)$.  It is easy to see that 
\begin{align*}
C&=(x_{2,1},...,x_{2,r},e(T_1),...,e(T_{M})) \\
&=(x_{2,1},...,x_{2,r})+\sum_{i=1}^{M}I_{T_i}.
\end{align*}
For the remainder of this paper we denote by $F^{G_1}$ and $F^C$, the minimal graded $Q$-free resolutions of $Q/I_{G_1}$ and $Q/C$, respectively.  The following fact is a key ingredient in our results.  

\begin{lemma}\label{lemmatensorresn}
The minimal graded free resolution of $Q/C$  over $Q$ is 
\begin{align*}
F^C=K(x_{2,1},...,x_{2,r};Q)\underset{Q}{\otimes}F^1\underset{Q}{\otimes}...\underset{Q}{\otimes}F^{M}
\end{align*}
where $F^i$ is the minimal graded free resolution of $Q/I_{T_i}$ for each $i$.
\end{lemma}   

\begin{proof}
Let $A$ denote the subring, $k[x_{2,1},...,x_{2,r}]$ of $Q$ and let $B$ denote the polynomial subring on all other variables in $Q$ so that $Q=A\otimes_k B$.  Let $\bar{K}$ be the minimal graded free resolution of $A/J$ over $A$, where $J=(x_{2,1},...,x_{2,r})$, and let $\bar{F}$ be the minimal graded free resolution of $B/L$ over $B$, where $L=\sum_{i=1}^{M}I_{T_i}$.  Then we have that $K=\bar{K}\otimes_k B$ and $F=A\otimes_k\bar{F}$ are minimal graded free resolutions of $Q/J$ and $Q/L$, respectively, over $Q$.  We note that $K$ is precisely the Koszul complex $K(x_{2,1},...,x_{2,r};Q)$.    
\newline
\indent  Now we have that 
\begin{align*}
K\underset{Q}{\otimes}F&=(\bar{K}\underset{k}{\otimes}B)\underset{Q}{\otimes}(A\underset{k}{\otimes}\bar{F})=(\bar{K}\underset{k}{\otimes}B)\underset{A\underset{k}{\otimes} B}{\otimes}(A\underset{k}{\otimes}\bar{F})=(\bar{K}\underset{A}{\otimes}A\underset{k}{\otimes}B)\underset{A\underset{k}{\otimes} B}{\otimes}(A\underset{k}{\otimes}B\underset{B}{\otimes}\bar{F}) \\
&=\bar{K}\underset{A}{\otimes}(A\underset{k}{\otimes}B)\underset{B}{\otimes}\bar{F}=\bar{K}\underset{k}{\otimes}\bar{F}.
\end{align*}
Thus, taking homology, we see that
\begin{align*}
H_n(K\underset{Q}{\otimes}F)=H_n(\bar{K}\underset{k}{\otimes}\bar{F})=\underset{i+j=n}{\bigoplus}\Big(H_i(\bar{K})\underset{k}{\otimes}H_{j}(\bar{F})\Big),
\end{align*}
where the last equality follows from the K{\"u}nneth Formula over $k$; see for example \cite[Cor 10.84]{rotman}.  Thus $K\otimes_Q F$ is exact in all positive degrees and $H_0(K\otimes_Q F)=A/J\otimes_k B/L=Q/(I+J)=Q/C$. Minimality is clear, so we have that
\begin{align*}
K(x_{2,1},...,x_{2,r};Q)\underset{Q}{\otimes}F
\end{align*}
is the minimal graded free resolution of $Q/C$.  Noticing that $T_1,...,T_{M}$ involve disjoint sets of variables, we can apply a similar argument repeatedly to conclude that $F\cong F^1\otimes_Q...\otimes_Q F^{M}$, thus giving the desired result.  
\end{proof} 

Now we have that the minimal graded free resolution $\mathbb{F}$ of $Q/I_G$ over $Q$ is the cone of $\phi$, where $\phi$ is a comparison map given by 

\begin{align}\label{eq:3}
\begin{tikzcd}[ampersand replacement=\&]
F^C(-2) \arrow{r}\arrow{d}{\phi} \& Q/C(-2) \arrow{d}{x_1x_2} \\
F^{G_1} \arrow{r} \& Q/I_{G_1} 
\end{tikzcd}
\end{align}

\noindent Thus, we have that $\mathbb{F}$ has modules $\mathbb{F}_i=F_i^{G_1}\oplus F_{i-1}^C(-2)$, and differentials,

\begin{align}
\partial_i=
\left[
\begin{array}{c c}
\partial_i^{G_1} & \phi_{i-1} \vspace{.12cm}\\
0 & -\partial_{i-1}^C \\
\end{array}
\right]. \label{eq:1}
\end{align}

\noindent We make the following remark about the resolutions $F^1,...,F^{M}$ in Lemma \ref{lemmatensorresn}.

\begin{remk}\label{remksubcx2}
The resolutions $F^1,...,F^{M}$ are subcomplexes of $F^{G_1}$.  Indeed, a minimal resolution of $Q/I_{G_1}$ can be obtained from each $F^q$ by the iterated mapping cone construction, thus by Remark \ref{remksubcx}, each $F^q$ is a subcomplex of $F^{G_1}$.
\end{remk} 

We now aim to give an explicit description of the map $\phi$.  To obtain such a description, we first observe that it is enough to define $\phi$ on elements of the form $\alpha\otimes 1\otimes...\otimes 1$.

\begin{lemma}\label{lemmaenoughalpha}
If $\widetilde{\phi}$ is a comparison map
\begin{center}
\begin{tikzcd}[ampersand replacement=\&]
K(x_{2,1},...,x_{2,r};Q)(-2) \arrow{r}\arrow{d}{\widetilde{\phi}} \& Q/(x_{2,1},...,x_{2,r})(-2) \arrow{d}{x_1x_2} \\
F^{G_1} \arrow{r} \& Q/I_{G_1} 
\end{tikzcd} 
\end{center}
then $\phi(\alpha\otimes\beta_1\otimes...\otimes\beta_{M})=\widetilde{\phi}(\alpha)\cdot\Big(\beta_1\cdot\big(...\cdot(\beta_{M-1}\cdot\beta_{M})...\big)\Big)$ defines a comparison map in (\ref{eq:3}).
\end{lemma}

Before giving a proof, we note that each $\beta_i$ is a basis element of $F^{G_1}$ by Remark \ref{remksubcx2}.  We also note that the multiplication appearing in the definition of $\phi$ in the lemma is a multiplication on the resolution $F^{G_1}$; it has one by Proposition \ref{propbe}.  Thus this definition of $\phi$ makes sense.  

\begin{proof}
We must check that $\phi$ is a chain map.  Thus we compute
\begin{align*}
&\phi(\partial^C(\alpha\otimes\beta_1\otimes...\otimes\beta_{M})) \\
&=\phi\Big(\partial^K(\alpha)\otimes\beta_1\otimes...\otimes\beta_{M}+\sum_{i=1}^{M}(-1)^{|\alpha|+...+|\beta_i|}\alpha\otimes\beta_1\otimes...\otimes\partial^{F^i}(\beta_i)\otimes...\otimes\beta_{M}\Big) \\
&=\widetilde{\phi}(\partial^K(\alpha))\cdot\Big(\beta_1\cdot\big(...\cdot(\beta_{M-1}\cdot\beta_{M})...\big)\Big) \\
&+\sum_{i=1}^{M}(-1)^{|\alpha|+...+|\beta_i|}\,\widetilde{\phi}(\alpha)\cdot\Bigg(\beta_1\cdot\Big(...\cdot\big(\partial^{F^i}(\beta_i)\cdot(...\cdot(\beta_{M-1}\cdot\beta_{M})...)\big)...\Big)\Bigg).
\end{align*}
On the other hand, we have that
\begin{align*}
&\partial^{G_1}(\phi(\alpha\otimes\beta_1\otimes...\otimes\beta_{M}))=\partial^{G_1}\Bigg(\widetilde{\phi}(\alpha)\cdot\Big(\beta_1\cdot\big(...\cdot(\beta_{M-1}\cdot\beta_{M})...\big)\Big)\Bigg) \\
&=\partial^{G_1}(\widetilde{\phi}(\alpha))\cdot\Big(\beta_1\cdot\big(...\cdot(\beta_{M-1}\cdot\beta_{M})...\big)\Big)+(-1)^{|\widetilde{\phi}(\alpha)|}\,\widetilde{\phi}(\alpha)\cdot\partial^{G_1}\Big(\beta_1\cdot\big(...\cdot(\beta_{M-1}\cdot\beta_{M})...\big)\Big).
\end{align*}
Applying the Leibniz rule repeatedly, and by our assumption that $\widetilde{\phi}$ is a chain map, we see that $\phi(\partial^C(\alpha\otimes\beta_1\otimes...\otimes\beta_{M}))=\partial^{G_1}(\phi(\alpha\otimes\beta_1\otimes...\otimes\beta_{M}))$, which completes the proof. 
\end{proof}

Now we work towards defining a comparison map $\widetilde{\phi}$.  To accomplish this, we need to examine $F^{G_1}$ more closely.  We apply the iterated mapping cone construction in the order given in (\ref{eq:4}).  We observe that the resolution of $Q/(x_2x_{2,1},...,x_2x_{2,r})$ produced by the iterated mapping cone procedure is precisely the Taylor resolution (see for example \cite[Constr 26.5]{peeva}), which we write as follows.  Let $E$ be the exterior algebra over $k$ on basis elements $e_1,...,e_r$.  Then the minimal graded free resolution of $Q/(x_2x_{2,1},...,x_2x_{2,r})$ over $Q$ is $F$, where $F_i=Q\otimes E_i$ and the differentials are given by
\begin{align*}
\partial^F(e_{j_1}...e_{j_i})&=\sum_{\ell=1}^i(-1)^{\ell-1}\frac{\mathrm{lcm}(x_2x_{2,j_1},...,x_2x_{2,j_i})}{\mathrm{lcm}(x_2x_{2,j_1},...\widehat{x_2x_{2,j_{\ell}}},...,x_2x_{2,j_i})}e_{j_1}...\widehat{e_{j_{\ell}}},...e_{j_i} \\ 
&=\begin{cases}
\sum_{\ell=1}^i(-1)^{\ell-1}x_{2,j_{\ell}}e_{j_1}...\widehat{e_{j_{\ell}}},...e_{j_i} & i\geq 2 \\ 
x_2x_{2,j_1} & i=1
\end{cases}.
\end{align*}  
Recall that, by Remark \ref{remksubcx}, $F$ is a subcomplex of $F^{G_1}$.  
\newline 
\indent We are now ready to define $\phi$.  To set up notation, we write the Koszul complex $K(x_{2,1},...,x_{2,r};Q)$ as $Q\langle a_1,...,a_r | \partial^K(a_j)=x_{2,j}\rangle$.  We note that if $F^{G_1}$ is a DG algebra, the result below follows from standard DG algebra results; see \cite[Prop 2.1.9]{MR2641236}.  Otherwise, a different proof is needed.

\begin{prop}\label{propphi}
Define $\phi:F^C(-2)\rightarrow F^{G_1}$ by 
\begin{align*}
\phi(a_{j_1}...a_{j_i}\otimes\beta_1\otimes...\otimes\beta_{M})=x_1e_{j_1}...e_{j_i}\cdot\Big(\beta_1\cdot\big(...\cdot(\beta_{M-1}\cdot\beta_{M})...\big)\Big).  
\end{align*}
Then $\phi$ is a comparison map for 
\begin{align*}
\begin{tikzcd}[ampersand replacement=\&]
F^C(-2) \arrow{r}\arrow{d}{\phi} \& Q/C(-2) \arrow{d}{x_1x_2} \\
F^{G_1} \arrow{r} \& Q/I_{G_1} 
\end{tikzcd}.
\end{align*}
\end{prop}

\begin{proof}
We define $\widetilde{\phi}(a_{j_1}...a_{j_i})=x_1e_{j_1}...e_{j_i}$ so that 
\begin{align*}
\phi(a_{j_1}...a_{j_i}\otimes\beta_1\otimes...\otimes\beta_{M})=\widetilde{\phi}(a_{j_1}...a_{j_i})\cdot\Big(\beta_1\cdot\big(...\cdot(\beta_{M-1}\cdot\beta_{M})...\big)\Big).
\end{align*}  
Then by Lemma \ref{lemmaenoughalpha}, it suffices to check that $\partial^{G_1}\widetilde{\phi}=\widetilde{\phi}\,\partial^{K}$.  For $i\geq 2$, we compute
\begin{align*}
\phi(\partial^C(a_{j_1}...a_{j_i}\otimes 1\otimes...\otimes 1))&=\phi\Big(\sum_{\ell=1}^i(-1)^{\ell-1}x_{2,j_{\ell}}a_{j_1}...\widehat{a_{j_{\ell}}}...a_{j_i}\otimes 1\otimes...\otimes 1\Big) \\
&=\sum_{\ell=1}^i(-1)^{\ell-1}x_1x_{2,j_{\ell}}e_{j_1}...\widehat{e_{j_{\ell}}}...e_{j_i} \\
&=\partial^{G_1}(x_1e_{j_1}...e_{j_i}) \\
&=\partial^{G_1}(\phi(a_{j_1}...a_{j_i}\otimes 1\otimes...\otimes 1)).
\end{align*}  
And, for any $j$, we have 
\begin{align*}
\widetilde{\phi}(\partial^K(a_j))=\widetilde{\phi}(x_{2,j})=x_1x_2x_{2,j}=\partial^{G_1}(x_1e_j)=\partial^{G_1}(\widetilde{\phi}(a_j))
\end{align*}
which completes the proof.
\end{proof}

To summarize the discussions in this section, we decompose $F^C(-2)$ and $F^{G_1}$ as 
\begin{align*}
F_i^C(-2)&=K_i(x_{2,1},...,x_{2,r};Q)(-2)\oplus\widetilde{F_i^C} \\
F_i^{G_1}&=F_i\oplus\widetilde{F_i^{G_1}}
\end{align*}
and think of the minimal graded free resolution $\mathbb{F}$ of $Q/I_G$ as the cone of the diagram,

\begin{align}\label{eq:12}
\begin{tikzcd}[ampersand replacement=\&]
\arrow{r} \& Q(-(r+2))\oplus\widetilde{F_{r}^C} \arrow{r}\arrow{ddd}{\begin{bmatrix}
  x_1 & *\\
  0 & *
\end{bmatrix}} \& \hspace{-2.5cm} \& \arrow{r} \& Q(-4)^{{r\choose 2}}\oplus\widetilde{F_2^C} \arrow{r}\arrow{ddd}{\begin{bmatrix}
  x_1 & *\\
  0 & *
\end{bmatrix}} \& Q(-3)^r\oplus\widetilde{F_1^C} \arrow{r}\arrow{ddd}{\begin{bmatrix}
  x_1 & *\\
  0 & *
\end{bmatrix}} \& Q(-2) \arrow{ddd}{x_1x_2} \\ \\ \\
\arrow{r} \& Q(-(r+1))\oplus\widetilde{F_r^{G_1}} \arrow{r} \& \hspace{-2.5cm} \& \arrow{r} \& Q(-3)^{{r\choose 2}}\oplus\widetilde{F_2^{G_1}} \arrow{r} \& Q(-2)^r\oplus\widetilde{F_1^{G_1}} \arrow{r} \& Q 
\end{tikzcd}
\end{align}

\noindent The following corollary is an immediate consequence of the construction and diagram above.

\begin{cor}\label{cornewlinstrgen}
For $\alpha\in K(x_{2,1},...,x_{2,r};Q)$ and $\ell=|\alpha|+1$, the elements $\Phi(\alpha\otimes 1\otimes...\otimes 1)$ are generators of $H_{\ell}(R)$ of internal degree $\ell+1$, thus they lie on the lowest linear strand.\hfill\qed
\end{cor}

We conclude this section by noting that the constructions above provide a way of counting the Betti numbers on the linear strand of $\mathbb{F}$, and equivalently the generators on the lowest linear strand of $H(R)$.  In particular, we recover the following result of Roth and Van Tuyl \cite[Cor 2.6]{rothvantuyl}.

\begin{cor}
Let $G$ be a tree.  Then $\beta_{1,2}(Q/I_G)=|e(G)|$ and 
\begin{align*}
\beta_{i,i+1}(Q/I_G)=\sum_{v\in G}{\mathrm{deg}(v) \choose i}
\end{align*}
for all $i\geq 2$.  
\end{cor}

\begin{proof}
We use induction on the number of edges in $G$.  For the base case we consider the tree with one edge.  In this case, $I_G=(x_1x_2)$ and the minimal graded free resolution $F$ of $R$ is 
\begin{align*}
0\longrightarrow Q(-2)\overset{x_1x_2}{\longrightarrow}Q\longrightarrow 0.
\end{align*}
Thus we see that $\beta_{1,2}(Q/I_G)=1=|e(G)|$.  
\newline
\indent  Now take $G$ to be any tree and assume that the result is true for every tree with strictly fewer edges.  We obtain the minimal graded free resolution $\mathbb{F}$ of $Q/I_G$ as the cone of the diagram (\ref{eq:12}) constructed in this section.  We count the Betti numbers on the linear strand as follows.  From (\ref{eq:12}) and Corollary \ref{cornewlinstrgen}, we see that for $i\geq 2$ we have
\begin{align*}
\beta_{i,i+1}(Q/I_G)\geq \beta_{i,i+1}(Q/I_{G_1})+{r \choose i-1}=\sum_{v\in G_1}{\mathrm{deg}_{G_1}(v) \choose i}+{r \choose i-1} 
\end{align*}   
where the equality follows from induction.  Separating the summand corresponding to $x_2$ from the rest of the sum, we get that 
\begin{align*}
\beta_{i,i+1}(Q/I_G)\geq\sum_{x_2\neq v\in G_1}{\mathrm{deg}(v) \choose i}+{r \choose i}+{r \choose i-1}=\sum_{x_2\neq v\in G_1}{\mathrm{deg}(v) \choose i}+{r+1 \choose i}
\end{align*}    
where the equality is just the identity called Pascal's Rule.  We note that $r+1$ is precisely the degree of $x_2$ in $G$.  Thus we have the inequality
\begin{align*}
\beta_{i,i+1}(Q/I_G)\geq \sum_{v\in G}{\mathrm{deg}(v) \choose i}.
\end{align*} 
\indent To show equality, it suffices to take $e$ to be any basis element of $\widetilde{F_{i-1}^C}$ and show that it cannot be on the linear strand.  By Lemma \ref{lemmatensorresn}, we have that $e=\alpha\otimes\beta_1\otimes...\otimes\beta_{M}$, for some basis elements $\alpha$ of $K_{\ell}(x_{2,1},...,x_{2,r})$ and $\beta_{p}$ of $(F^p)_{i_p}$, where $\ell+i_1+...+i_{M}=i-1$ and at least one $\beta_{p}\neq 1$.  In the following computations, we denote by $|\cdot|$ the homological degree and by $\mathrm{intdeg}(\cdot)$ the internal degree in $\mathbb{F}$.  We denote by $\mathrm{intdeg}_C(\cdot)$ the internal degree in $F^C$.   We have that
\begin{align*}
\mathrm{intdeg}(e)&=\mathrm{intdeg}_C(e)+2 \\
&=\mathrm{intdeg}_C(\alpha\otimes 1\otimes...\otimes 1)+\mathrm{intdeg}_C(1\otimes\beta_1\otimes...\otimes 1)+...+\mathrm{intdeg}_C(1\otimes 1\otimes ...\otimes\beta_M)+2 \\
&=\mathrm{intdeg}(\alpha\otimes 1\otimes...\otimes 1)+\mathrm{intdeg}(\beta_1)+...+\mathrm{intdeg}(\beta_M)
\end{align*}
where the first equality follows from the fact that $\mathbb{F}=\mathrm{cone}(F^C(-2)\rightarrow F^{G_1})$, and the last equality follows from this same fact and also from Remark \ref{remksubcx2}.  Now by minimality we have that
\begin{align*}
\mathrm{intdeg}(e)&\geq|\alpha\otimes 1\otimes...\otimes 1|+|\beta_1|+...+|\beta_M|+M+1 \\
&=\ell+i_1+...+i_M+M+2 \\
&=i+M+1 \\
&>i+1.
\end{align*}
Therefore, since $|e|=i$, $e$ cannot possibly be on the linear strand, and we have equality.  The desired formula for $\beta_{1,2}(Q/I_G)$ clearly holds.
\end{proof}

\section{The Main Result}
\paragraph
\indent  In this section, we show that Question \ref{quest} has a positive answer for edge ideals of trees, and thus also for forests.  Throughout this section we let $G$ be a tree and let $Q$ be a standard graded polynomial ring over any field $k$ with variables given by the vertices in $G$.  Denote by $N$ the set indices for the vertices, so that $Q=k[\{x_n\}_{n\in N}]$.  Let $\Phi_{G_1}$ be the isomorphism $F^{G_1}\otimes_Q k\rightarrow H(Q/I_{G_1})$ as in (\ref{eq:6}).      

\begin{lemma}\label{lemmainducedalgmap}
The canonical map of $k$-algebras
\begin{align*}
\theta:H(Q/I_{G_1})\longrightarrow H(R)
\end{align*}
induced by the surjection $Q/I_{G_1}\rightarrow Q/I_G=R$ satisfies the equality $\theta(\Phi_{G_1}(e\otimes\bar{1}))=\Phi(e\otimes\bar{1})$ for any $e\in F^{G_1}$.
\end{lemma}

\begin{proof}
First recall that the quotient map
\begin{align*}
Q/I_G\longrightarrow Q/I=R
\end{align*}
induces the map of DG algebras 
\begin{align*}
K(Q/I_G)\longrightarrow K(R)
\end{align*}
that sends $dx_i$ to $dx_i$ for all $i$, and the induced map 
\begin{align*}
\theta:H(Q/I_{G_1})\longrightarrow H(R)
\end{align*}
on homology is a map of $k$-algebras.  Since $F^{G_1}$ is a subcomplex of $\mathbb{F}$ (see Remark \ref{remksubcx}), it is clear that the equality $\theta(\Phi_{G_1}(e\otimes\bar{1}))=\Phi(e\otimes\bar{1})$ holds.  
\end{proof}

Now we show that the Koszul homology algebra of the quotient by $I_G$ is generated by the lowest linear strand.

\begin{thm}\label{thmmain}
If $R=Q/I_G$, then $H(R)$ is generated by $\bigoplus_i H_i(R)_{i+1}$ as a $k$-algebra.
\end{thm}

\begin{proof}
We use induction on the number of edges in $G$.  For the base case, we consider the tree with one edge.  In this case, $I_G=(x_1x_2)$ and the minimal graded free resolution $F$ of $R$ is 
\begin{align*}
0\longrightarrow Q(-2)\overset{x_1x_2}{\longrightarrow}Q\longrightarrow 0.
\end{align*}
Thus applying the isomorphism $\Phi$ from (\ref{eq:6}) to $F\otimes k$, we see that the only basis element of $H_1(R)$ lies in $H_1(R)_2$.  Hence $H(R)$ is trivially generated by the lowest linear strand.
\newline
\indent  Now take $G$ to be any tree and assume that the result is true for every tree with strictly fewer edges.  Let $\mathbb{F}$ be the minimal graded resolution of $R$ over $Q$ constructed in Section 3 and fix the basis of each $H_i(R)$ given in Theorem \ref{thmhm}.  It is enough to show that each basis element of $H_i(R)$ is in the subalgebra generated by $\bigoplus_j H_j(R)_{j+1}$.  Thus, we take $h$ to be any basis element of $H_i(R)$.  Then $h=\Phi(e\otimes\bar{1})$, for some basis element $e$ of $\mathbb{F}_i$.  We have that $\mathbb{F}$ is the cone of the map
\begin{align*}
F^C\overset{\phi}{\longrightarrow}F^{G_1}
\end{align*}
defined in Proposition \ref{propphi}.  Thus $\mathbb{F}_i=F_i^{G_1}\oplus F_{i-1}^C$ and $e$ must either be a basis element of $F_i^{G_1}$ or of $F_{i-1}^C$.  
\newline
\indent  We first consider the case where $e$ is a basis element of $F_i^{G_1}$.  By Lemma \ref{lemmainducedalgmap}, $h=\theta(\Phi_{G_1}(e\otimes\bar{1}))$, but by the induction hypothesis, $H(Q/I_{G_1})$ is generated by the lowest linear strand.  So we have that 
\begin{align*}
\Phi_{G_1}(e\otimes\bar{1})=\sum_{\lambda\in\Lambda} c_{\lambda} \prod_{\ell,m} \Phi_{G_1}(e_{\ell}^{m,m+1}\otimes\bar{1})^{\lambda_{\ell,m}} 
\end{align*}
where $\Lambda$ is a finite set of tuples $\lambda=(\lambda_{\ell,m})$ and for $\ell=1,...,b_{m,m+1}$ the elements $e_{\ell}^{m,m+1}$ are basis elements of $F_m^{G_1}$ of internal degree $m+1$, and where $c_{\lambda}\in k$.  Now we have that 
\begin{align*}
h&=\theta\left(\sum_{\lambda\in\Lambda} c_{\lambda} \prod_{\ell,m} \Phi_{G_1}(e_{\ell}^{m,m+1}\otimes\bar{1})^{\lambda_{\ell,m}}\right)=\sum_{\lambda\in\Lambda} c_{\lambda} \prod_{\ell,m}\theta(\Phi_{G_1}(e_{\ell}^{m,m+1}\otimes\bar{1}))^{\lambda_{\ell,m}} \\
&=\sum_{\lambda\in\Lambda} c_{\lambda} \prod_{\ell,m}\Phi(e_{\ell}^{m,m+1}\otimes\bar{1})^{\lambda_{\ell,m}}
\end{align*} 
by Lemma \ref{lemmainducedalgmap}.  The elements $\Phi(e_{\ell}^{m,m+1}\otimes\bar{1})$ are basis elements of $H_m(R)$ that are in $H_m(R)_{m+1}$, thus $h$ is generated in the subalgebra generated by the linear strand.
\newline
\indent Now we assume that $e$ is a basis element of $F_{i-1}^C$.  Then, by Lemma \ref{lemmatensorresn}, $e=\alpha\otimes\beta_1\otimes...\otimes\beta_{M}$, for some basis elements $\alpha$ of $K_{\ell}(x_{2,1},...,x_{2,r})$ and $\beta_{p}$ of $(F^p)_{i_p}$, where $\ell+i_1+...+i_{M}=i-1$.  By Theorem \ref{thmhm}, $h=[\bar{g}]$, where
\begin{align}\label{eq:5}
g=\sum_{\{k_1<...<k_i\}\subseteq N}\sum_{j_1=1}^{b_1}...\sum_{j_{i-1}=1}^{b_{i-1}}d^{k_i}(f_{j_{i-1},j_i}^i)...d^{k_2}(f_{j_1,j_2}^2)d^{k_1}(f_{1,j_1}^1)dx_{k_1}...dx_{k_i}
\end{align}
and each $f_{i,j}^k$ is the $(i,j)$-th entry in the $k$th differential of $\mathbb{F}$ when viewing the differentials as matrices with respect some fixed bases.  
\newline
\indent  Recall that, the operators $d^k$ depend on the order of the variables, so we fix an ordering on the variables in $Q$ as follows
\begin{align*}
x_1<x_2<x_{2,\ell}<v(T_1)<...<v(T_{M})
\end{align*} 
for all $\ell$, where by $v(T_q)$ we mean the variables given by the vertices in the tree $T_q$ listed in some fixed order.  
\newline 
\indent Now we analyze the terms in (\ref{eq:5}) more carefully in order to remove the initial sum.  We note that since the differentials of $\mathbb{F}$ are given by (\ref{eq:1}), we have that for each set $\{j_1,...,j_{i-1}\}$, there is some $m$ such that $f_{j_{i-1},j_i}^i,...,f_{j_{m},j_{m+1}}^{m+1}$ are entries of $\partial^C$, $f_{j_{m-1},j_m}^m$ is an entry of $\phi_{m-1}$, and $f_{j_{m-2},j_{m-1}}^{m-1},...,f_{1,j_{1}}^{1}$ are entries of $\partial^{G_1}$.  That is, for each $p=1,...,m-1$, we have that $1\leq j_p\leq \beta_p(Q/I_{G_1})$, and for $p=m,...,i$, we have $\beta_p(Q/I_{G_1})+1\leq j_p\leq b_p$.  Then, in particular, we have that $f_{j_{m-1},j_m}^m\in (x_1)$ by Proposition \ref{propphi}.  Thus by Lemma \ref{lemmaoperators}, we have that $d^{k_m}(f_{j_{m-1},j_m}^m)=0$ unless $k_m=1$.  Also since $k_1<...<k_m<...<k_i$, our fixed ordering on the variables implies that $m=1$.  So, every term in the sum with $k_m\neq 1$ or $m\neq 1$ vanishes, giving
\begin{align}
g=\sum_{\{1<k_2<...<k_i\}\subseteq N}\sum_{j_1=a_1}^{b_1}...\sum_{j_{i-1}=a_{i-1}}^{b_{i-1}}d^{k_i}(f_{j_{i-1},j_i}^i)...d^{k_2}(f_{j_1,j_2}^2)d^{1}(f_{1,j_1}^1)dx_{1}dx_{k_2}...dx_{k_i}
\end{align}
where $a_p=\beta_p(Q/I_{G_1})+1$, and with $f_{j_{i-1},j_i}^i,...,f_{j_{1},j_{2}}^{2}$ entries of $\partial^C$, hence monomials, and $f_{1,j_1}^1$ an entry of $\phi_0$.  We note that $b_1=a_1$ since $F_0^C=Q$, thus there is only one possible index $j_1$, and the sum over $j_1$ can be removed.  Also, $\phi_0$ is given by multiplication by $x_1x_2$, thus $f_{1,j_1}^1=x_1x_2$.
\newline
\indent  Similarly by Lemma \ref{lemmaoperators}, for each $f_{j_{p-1},j_p}^p$ with $p>1$, there is only one value of $k_p$ such that $d^{k_p}(f_{j_{p-1},j_p}^p)$ is nonzero.  Thus, we see that for each set $\{j_1,...,j_{i-1}\}$, there exist unique $k_p=k_q(j_1,...,j_{i-1})$, for $p=2,...,i$, such that  
\begin{align}
g=\sum_{j_2=a_2}^{b_2}...\sum_{j_{i-1}=a_{i-1}}^{b_{i-1}}d^{k_i}(f_{j_{i-1},j_i}^i)...d^{k_2}(f_{j_1,j_2}^2)d^{1}(x_1x_2)dx_{1}dx_{k_2}...dx_{k_i} \label{eq:2}
\end{align}
with $f_{j_{i-1},j_i}^i,...,f_{j_{1},j_{2}}^{2}$ entries of $\partial^C$.  For ease of exposition in the rest of the proof, we drop the bounds on the sums in (\ref{eq:2}), and we denote by $\sum_{j_p}$ the sum from $j_p=a_p$ to $j_p=b_p$, for each $p$.   
\newline
\indent Furthermore, since 
\begin{align*}
&\partial^C(\alpha\otimes\beta_1\otimes...\otimes\beta_{M})=\partial^K(\alpha)\otimes\beta_1\otimes...\otimes\beta_{M}+\sum_{i=1}^{M}(-1)^{|\alpha|+...+|\beta_i|}\alpha\otimes\beta_1\otimes...\otimes\partial^{F^i}(\beta_i)\otimes...\otimes\beta_{M},
\end{align*}  
our chosen order of variables and Lemma \ref{lemmaoperators} imply that the only nonzero terms in (\ref{eq:2}) are the ones such that 
\begin{align*}
f_{j_{i-1},j_i}^i,...,f_{j_{i-i_{M}},j_{i-i_{M}+1}}^{i-i_{M}+1} \,\,\,&\text{are entries of}\,\,\, \partial^{F^{M}} \\
f_{j_{i-i_{M}-1},j_{i-i_{M}}}^{i-i_{M}},...,f_{j_{i-i_{M}-i_{M-1}},j_{i-i_{M}-i_{M-1}+1}}^{i-i_{M}-i_{M-1}+1}\,\,\,&\text{are entries of}\,\,\, \partial^{F^{M-1}} \\
\vdots \\
f_{j_{\ell+i_1},j_{\ell+i_1+1}}^{\ell+i_1+1},...,f_{j_{\ell+1},j_{\ell+2}}^{\ell+2} \,\,\,&\text{are entries of}\,\,\, \partial^{F^1} \\
f_{j_{\ell},j_{\ell+1}}^{\ell+1},...,f_{j_1,j_2}^{2} \,\,\,&\text{are entries of}\,\,\, \partial^{K}.
\end{align*}\label{eq:7}
where here we identify each $\partial^{F^p}$ with a matrix with respect to our fixed bases.  Notice that (\ref{eq:2}) can be written as
\begin{align}\label{eq:8}
g=\sum_{j_{\ell+1}}...\sum_{j_{i-1}}d^{k_i}(f_{j_{i-1},j_i}^i)...d^{k_{\ell+2}}(f_{j_{\ell+1},j_{\ell+2}}^{\ell+2})\left(\sum_{j_2}...\sum_{j_{\ell}}d^{k_{\ell+1}}(f_{j_{\ell},j_{\ell+1}}^{\ell+1})...d^{1}(x_1x_2)dx_{1}...dx_{k_{\ell+1}}\right)dx_{k_{\ell+2}}...dx_{k_i}.
\end{align}
Next note that the sum over $j_{\ell+1}$ can be removed.  Indeed $f_{j_{\ell+1},j_{\ell+2}}^{\ell+2}$ is an entry of $\partial_{1}^{F^1}$ and since $F^1$ is the minimal graded free resolution of a cyclic module only one index $j_{\ell+1}$ is possible.  The last group of sums in (\ref{eq:8}) can be pulled out to yield  
\begin{align*}
g=\left(\sum_{j_2}...\sum_{j_{\ell}}d^{k_{\ell+1}}(f_{j_{\ell},j_{\ell+1}}^{\ell+1})...d^{1}(x_1x_2)dx_{1}...dx_{k_{\ell+1}}\right)\left(\sum_{j_{\ell+2}}...\sum_{j_{i-1}}d^{k_i}(f_{j_{i-1},j_i}^i)...d^{k_{\ell+2}}(f_{j_{\ell+1},j_{\ell+2}}^{\ell+2})dx_{k_{\ell+2}}...dx_{k_i}\right).
\end{align*}
We observe that the homology class of the first factor is precisely $\Phi(\alpha\otimes\bar{1})$.  Repeating the procedure above with the sums in the second factor and then taking homology classes, we find that
\begin{align*}
h=\Phi(\alpha\otimes\bar{1})\cdot\Phi(\beta_{1}\otimes\bar{1})\cdot....\cdot\Phi(\beta_{M}\otimes\bar{1})
\end{align*} 
by Remark \ref{remksubcx2}.  Since $\beta_1,...,\beta_{M}$ are basis elements of $F^{G_1}$, a similar inductive argument to the one given in the first case implies that
\begin{align*}
h&=\Phi(\alpha\otimes\bar{1})\cdot\left(\sum_{\lambda\in\Lambda} c_{\lambda} \prod_{\ell,m}\Phi((\beta_{1})_{\ell}^{m,m+1}\otimes\bar{1})^{\lambda_{\ell,m}}\right)\cdot...\cdot\left(\sum_{\lambda\in\Lambda} c_{\lambda} \prod_{\ell,m}\Phi((\beta_{M})_{\ell}^{m,m+1}\otimes\bar{1})^{\lambda_{\ell,m}}\right)
\end{align*}  
where each $(\beta_{q})_{\ell}^{m,m+1}$ is a basis element of $F_m$ of internal degree $m+1$.  In addition, by Corollary \ref{cornewlinstrgen}, $\Phi(\alpha\otimes\bar{1})$ is a generator of $H_{\ell+1}(R)$ that is in $H_{\ell+1}(R)_{\ell+2}$.  Therefore $h$ is in the subalgebra generated by the lowest linear strand.
\end{proof}

Since paths are trees, this recovers \cite[Thm 3.15]{boocheretal}.  Now consider a forest $G$.  By definition, $G$ is a disjoint union of trees, $T_1$,...,$T_m$.  Thus, the quotient of the edge ideal of $G$ is of the form
\begin{align*}
Q/I_G=Q_1/I_{T_1}\underset{k}{\otimes}...\underset{k}{\otimes}Q_m/I_{T_m}
\end{align*}
where $Q_1,...,Q_m$ are polynomial rings on disjoint sets of variables such that $Q=Q_1\otimes_k...\otimes_k Q_m$.  This induces an isomorphism on the Koszul homology algebras
\begin{align*}
H(Q/I_G)\cong H(Q_1/I_{T_1})\underset{k}{\otimes}...\underset{k}{\otimes}H(Q_m/I_{T_m}),
\end{align*}
thus yielding the following corollary as a direct consequence of Theorem \ref{thmmain}.

\begin{cor}
If $R=Q/I_G$ and $G$ is a forest, then $H(R)$ is generated by $\bigoplus_i H_i(R)_{i+1}$ as $k$-algebras.\hfill\qed
\end{cor}

\section*{\centering\normalsize{\normalfont{\uppercase{Acknowledgments}}}}
The author thanks her advisor, Claudia Miller, for guidance and support and Srikanth Iyengar for suggesting studying Question \ref{quest}. The author is also grateful for helpful conversations with Aldo Conca, Eloisa Grifo, Adam Boocher, and Matthew Mastroeni in the early stages of this project.

\bibliographystyle{siam}
\bibliography{edgeidealsv12}
\nocite{*}

\vspace{1.25cm}

\noindent Department of Mathematics, Syracuse University, Syracuse, NY, 13244
\newline
\textit{email}: rngettin@syr.edu

 \end{document}